\documentclass[12p]{amsart}
\usepackage{amssymb}
\usepackage{amsmath}
\usepackage{amsfonts}
\usepackage{geometry}
\usepackage{graphicx}
\usepackage{mathrsfs,amssymb}

\usepackage{hyperref}
\usepackage{cleveref}

\theoremstyle{plain}

\newtheorem{lemma}{Lemma}

\newtheorem{remark}{Remark}

\newtheorem{theorem}{Theorem}
\numberwithin{equation}{section}

\begin{document}

\title{Spacetime integral bounds for the energy-critical nonlinear wave equation}
\date{\today}
\author{Benjamin Dodson}
\maketitle

\begin{abstract}
In this paper we prove a global spacetime bound for the quintic, nonlinear wave equation in three dimensions. This bound depends on the $L_{t}^{\infty} L_{x}^{2}$ and $L_{t}^{\infty} \dot{H}^{2}$ norms of the solution to the quintic problem.
\end{abstract}

\section{Introduction}
It has been known for a long time that the defocusing, quintic nonlinear wave equation,
\begin{equation}\label{1.1}
u_{tt} - \Delta u + u^{5} = 0, \qquad u(0,x) = u_{0}, \qquad u_{t}(0,x) = u_{1}, \qquad u : I \times \mathbb{R}^{3} \rightarrow \mathbb{R},
\end{equation}
has a global solution that scatters for initial data lying in the critical Sobolev space $(u_{0}, u_{1}) \in \dot{H}^{1} \cap L^{6} \times L^{2}$. Equation $(\ref{1.1})$ is called energy-critical because the solution conserves the energy
\begin{equation}\label{1.2}
E(u(t)) = \frac{1}{2} \int |\nabla u(t,x)|^{2} dx + \frac{1}{2} \int |u_{t}(t,x)|^{2} dx + \frac{1}{6} \int |u(t,x)|^{6} dx,
\end{equation}
which is invariant under the scaling symmetry
\begin{equation}\label{1.14}
u(t,x) \mapsto \lambda^{1/2} u(\lambda t, \lambda x).
\end{equation}
Conservation of energy implies that for all $t \in I$, where $I$ is the maximal interval of existence for a solution to $(\ref{1.1})$,
\begin{equation}\label{1.3}
\| (u(t), u_{t}(t)) \|_{\dot{H}^{1} \cap L^{2} \times L^{2}} \lesssim_{\| u_{0} \|_{\dot{H}^{1} \cap L^{6}}, \| u_{1} \|_{L^{2}}} 1.
\end{equation}
The Sobolev embedding theorem guarantees that in $\mathbb{R}^{3}$, $\| u_{0} \|_{L^{6}} \lesssim \| u_{0} \|_{\dot{H}^{1}}$.\medskip

Equation $(\ref{1.1})$ belongs to the large class of equations for which the local and small data theory is entirely determined by its scaling symmetry, $(\ref{1.14})$. Indeed, the local well-posedness result of \cite{MR153967} implies that $I$ is an open interval for any $(u_{0}, u_{1}) \in \dot{H}^{1} \times L^{2}$. Scattering was proved in \cite{MR731347} and  \cite{MR631403} for small energy data. Observe that $u(t,x)$ is a solution to $(\ref{1.1})$ if and only if $v(t,x) = \lambda^{1/2} u(\lambda t, \lambda x)$ is a solution to $(\ref{1.1})$ for different initial data.\medskip

For radial initial data, \cite{MR1015805} proved global well-posedness of $(\ref{1.1})$ for any initial data lying in the energy space. Later, \cite{ginibre1992global} combined the Morawetz estimate
\begin{equation}\label{1.4}
\int_{I} \int \frac{1}{|x|} u(t,x)^{6} dx dt \lesssim E(u_{0}, u_{1}),
\end{equation}
with the radial Sobolev embedding theorem,
\begin{equation}\label{1.5}
\| |x| u^{2} \|_{L_{t,x}^{\infty}(I \times \mathbb{R}^{3})} \lesssim E(u_{0}, u_{1}),
\end{equation}
to obtain the bound
\begin{equation}\label{1.6}
\int_{I} \int_{\mathbb{R}^{3}} u(t,x)^{8} dx dt \lesssim E(u_{0}, u_{1})^{2}.
\end{equation}
This bound is enough to prove that $(\ref{1.1})$ is globally well-posed and scattering for any radial initial data $(u_{0}, u_{1}) \in \dot{H}^{1} \times L^{2}$. Scattering is defined as there exist $(u_{0}^{+}, u_{1}^{+}) \in \dot{H}^{1} \times L^{2}$ and $(u_{0}^{-}, u_{1}^{-}) \in \dot{H}^{1} \times L^{2}$ such that
\begin{equation}\label{1.7}
\lim_{t \rightarrow +\infty} \| (u(t), u_{t}(t)) - S(t)(u_{0}^{+}, u_{1}^{+}) \|_{\dot{H}^{1} \times L^{2}} = 0,
\end{equation}
and
\begin{equation}\label{1.8}
\lim_{t \rightarrow -\infty} \| (u(t), u_{t}(t)) - S(t)(u_{0}^{-}, u_{1}^{-}) \|_{\dot{H}^{1} \times L^{2}} = 0,
\end{equation}
where $S(t)$ is the solution operator to the free wave equation $u_{tt} - \Delta u = 0$.

For non-radial initial data, \cite{grillakis1990regularity} proved global well-posedness and persistence of regularity for $(\ref{1.1})$ with smooth initial data. This result was extended to higher dimensions by \cite{grillakis1992regularity} and \cite{MR1266760}. Later, \cite{MR1283026} proved global well-posedness for initial data in the energy space. Using the profile decomposition of \cite{bahouri1999high}, \cite{MR1671997} proved global well-posedness and scattering for the quintic nonlinear wave equation $(\ref{1.1})$ with initial data in the energy space. See also \cite{MR1666973} for the Klein-Gordon equation. See \cite{bahouri1998decay} for a decay estimate for $\| u(t) \|_{L^{6}}$.

The proof of scattering is equivalent to the proof that for any $(u_{0}, u_{1}) \in \dot{H}^{1} \times L^{2}$, $(\ref{1.1})$ has a global solution $u$ that satisfies the bound
\begin{equation}\label{1.8.1}
\| u \|_{L_{t,x}^{8}(\mathbb{R} \times \mathbb{R}^{3})} \lesssim_{E(u_{0}, u_{1})} 1.
\end{equation}
(The argument proving this fact may be found in \cite{dodson2018global}, for example.) However, for non-radial data, the best spacetime bounds for the scattering size are much weaker than the bounds for the radial data, $(\ref{1.6})$.
\begin{theorem}\label{t1.1}
Let $(u_{0}, u_{1}) \in \dot{H}_{x}^{1}(\mathbb{R}^{3}) \times L_{x}^{2}(\mathbb{R}^{3})$ be initial data with the energy bound
\begin{equation}\label{1.9}
\int_{\mathbb{R}^{3}} \frac{1}{2} |u_{1}|^{2} + \frac{1}{2} |\nabla u_{0}|^{2} + \frac{1}{6} |u_{0}|^{6} dx \leq E.
\end{equation}
Then there exists a unique global solution $u \in C_{t}^{0} H_{x}^{1}(\mathbb{R} \times \mathbb{R}^{3}) \cap C_{t}^{1} L_{x}^{2}(\mathbb{R} \times \mathbb{R}^{3}) \cap L_{t}^{4} L_{x}^{12}(\mathbb{R} \times \mathbb{R}^{3})$ with the spacetime bound
\begin{equation}\label{1.10}
\| u \|_{L_{t}^{4} L_{x}^{12}(\mathbb{R} \times \mathbb{R}^{3})} \leq C(1 + E)^{C E^{105/2}},
\end{equation}
for some absolute constant $C > 0$. Interpolating $(\ref{1.10})$ with $\| u \|_{L_{t}^{\infty} L_{x}^{6}(\mathbb{R} \times \mathbb{R}^{3})} \leq E^{1/6}$ gives
\begin{equation}\label{1.10.1}
\| u \|_{L_{t,x}^{8}(\mathbb{R} \times \mathbb{R}^{3})} \leq C(1 + E)^{C E^{105/2}},
\end{equation}
\end{theorem}
This theorem was proved in \cite{tao2006spacetime}. The proof used the induction on energy argument. This argument was previously used in \cite{bourgain1999global} to prove scattering for the radially symmetric, nonlinear Schr{\"o}dinger equation in dimensions three and four. This argument was also used in \cite{tao2004global}, proving scattering in dimensions five and higher for the radially symmetric nonlinear Schr{\"o}dinger equation. Explicit scattering size bounds were also obtained in \cite{tao2004global}.\medskip

In this note we obtain the following bounds for the size of the spacetime integral of a solution to $(\ref{1.1})$.
\begin{theorem}\label{t1.2}
Let $(u_{0}, u_{1}) \in \dot{H}_{x}^{1}(\mathbb{R}^{3}) \times L_{x}^{2}(\mathbb{R}^{3})$ be initial data with the energy bound
\begin{equation}\label{1.11}
\int_{\mathbb{R}^{3}} \frac{1}{2} |u_{1}|^{2} + \frac{1}{2} |\nabla u_{0}|^{2} + \frac{1}{6} |u_{0}|^{6} dx \leq E.
\end{equation}
Also suppose that the solution to $(\ref{1.1})$ has the a priori bounds
\begin{equation}\label{1.12}
\| (u, u_{t}) \|_{L_{t}^{\infty} L_{x}^{2} \times \dot{H}^{-1}(\mathbb{R} \times \mathbb{R}^{3})} \leq A, \qquad \| (u, u_{t}) \|_{L_{t}^{\infty} \dot{H}^{2} \times \dot{H}^{1}(\mathbb{R} \times \mathbb{R}^{3})} \leq A,
\end{equation}
for some $A < \infty$. Then there exists a unique global solution $u \in C_{t}^{0} H_{x}^{1}(\mathbb{R} \times \mathbb{R}^{3}) \cap C_{t}^{1} L_{x}^{2}(\mathbb{R} \times \mathbb{R}^{3}) \cap L_{t,x}^{8}(\mathbb{R} \times \mathbb{R}^{3})$ with the spacetime bound
\begin{equation}\label{1.13}
\| u \|_{L_{t,x}^{8}(\mathbb{R} \times \mathbb{R}^{3})} \leq C E^{4/7} A \exp(C E^{85/6} E^{13/14} A^{11}),
\end{equation}
for some absolute constant $C > 0$.
\end{theorem}
\begin{remark}
One may use the scaling symmetry $(\ref{1.14})$ to obtain the equality that arises in $(\ref{1.12})$.
\end{remark}

The proof of Theorem $\ref{t1.2}$ follows the line of argument previously used for the nonlinear Schr{\"o}dinger equation in \cite{dodson2017new}, \cite{arora2020scattering}, and especially in \cite{dodson2017new1}. We utilize an interaction Morawetz estimate to show that the energy of a solution must eventually spread out in $\mathbb{R}^{3}$. More precisely, for any $\epsilon > 0$ and $T < \infty$, there exists $\mathcal T(T, \epsilon) < \infty$ such that for some $t_{0} \in [0, \mathcal T(T, \epsilon)]$,
\begin{equation}\label{1.15}
\int_{t_{0}}^{t_{0} + T} \int |u(t,x)|^{8} dx dt \leq \epsilon.
\end{equation}
Combining this fact with some dispersive estimates for the linear wave equation implies Theorem $\ref{t1.2}$.

\section{Local well-posedness and small data arguments}
As was mentioned in the introduction, \cite{grillakis1990regularity} proved that $(\ref{1.1})$ is globally well-posed for initial data in the energy space. Additionally, the a priori bound
\begin{equation}\label{2.1}
\| (u, u_{t}) \|_{L_{t}^{\infty} \dot{H}^{2} \times \dot{H}^{1}(\mathbb{R} \times \mathbb{R}^{3})} \leq A,
\end{equation}
implies that 
\begin{theorem}\label{t2.1}
For any $t_{0} \in \mathbb{R}$ and $T > 0$,
\begin{equation}\label{2.2}
\| u \|_{L_{t,x}^{8}([0, T] \times \mathbb{R}^{3})}^{8} \lesssim A^{8} \langle T \rangle.
\end{equation}
\end{theorem}
\begin{proof}
This follows from the Sobolev embedding theorem. Indeed, since $S(t)$ is a unitary operator on $\dot{H}^{s} \times \dot{H}^{s - 1}$, for any $s \in \mathbb{R}$, the bounds $(\ref{2.1})$ and $(\ref{1.12})$ imply
\begin{equation}\label{2.3}
\| S(t - t_{0})(u(t_{0}), u_{t}(t_{0})) \|_{L^{8}} \lesssim \| (u(t_{0}), u_{t}(t_{0})) \|_{\dot{H}^{9/8} \times \dot{H}^{1/8}} \lesssim A.
\end{equation}
Taking $|I|$ sufficiently small depending on $E$ and $A$, small data arguments imply
\begin{equation}\label{2.4}
\| u \|_{L_{t,x}^{8}(I \times \mathbb{R}^{3})} \lesssim 1.
\end{equation}
Since the bound $(\ref{2.1})$ is uniform on $\mathbb{R}$, we can partition $[t_{0}, t_{0} + T]$ into $\lesssim A^{8} \langle T \rangle$ such intervals for which $(\ref{2.4})$, which proves $(\ref{2.2})$.
\end{proof}
\begin{remark}
In fact, for any $\dot{H}^{1}$-admissible pair $(p, q)$ in $\mathbb{R}^{3}$, we have proved
\begin{equation}\label{2.4.1}
\| u \|_{L_{t}^{p} L_{x}^{q}([t_{0}, t_{0} + T] \times \mathbb{R}^{3})} \lesssim A \langle T \rangle^{1/p}.
\end{equation}
\end{remark}

We can prove a small data well-posedness result.
\begin{theorem}\label{t2.2}
There exists $\epsilon(E) > 0$ such that if $I = [t_{0}, t_{0} + T]$ is an interval for which
\begin{equation}\label{2.5}
\| S(t - t_{0})(u(t_{0}), u_{t}(t_{0})) \|_{L_{t,x}^{8}(I \times \mathbb{R}^{3})} \leq \epsilon,
\end{equation}
and 
\begin{equation}\label{2.5.1}
\| u \|_{L_{t}^{\infty} \dot{H}^{1}(I \times \mathbb{R}^{3})} \leq E^{1/2},
\end{equation}
then
\begin{equation}\label{2.6}
\| u \|_{L_{t,x}^{8}(I \times \mathbb{R}^{3})} \leq 2 \epsilon.
\end{equation}
\end{theorem}
\begin{proof}
By Duhamel's principle, for any $t \in I$,
\begin{equation}\label{2.7}
u(t) = S(t - t_{0})(u(t_{0}), u_{t}(t_{0})) - \int_{t_{0}}^{t} S(t - \tau)(0, u^{5}) d\tau.
\end{equation}
By interpolation,
\begin{equation}\label{2.8}
\| |\nabla|^{1/2} u^{5} \|_{L_{t}^{16/9} L_{x}^{16/13}(I \times \mathbb{R}^{3})} \lesssim \| \nabla u \|_{L_{t}^{\infty} L_{x}^{2}}^{1/2} \| u \|_{L_{t,x}^{8}(I \times \mathbb{R}^{3})}^{9/2}.
\end{equation}
Then interpolating the estimates,
\begin{equation}\label{2.9}
\| S(t)(0, F) \|_{L_{x}^{\infty}} \lesssim \frac{1}{t} \| \nabla F \|_{L^{1}}, \qquad \| S(t)(0, F) \|_{L_{x}^{2}} \lesssim \| |\nabla|^{-1} F \|_{L^{2}},
\end{equation}
the Hardy-Littlewood-Sobolev inequality implies
\begin{equation}\label{2.10}
\| |\nabla|^{1/4}, |\partial_{t}|^{1/4} \int_{t_{0}}^{t} S(t - \tau)(0, u^{5}) d\tau \|_{L_{t,x}^{16/3}(I \times \mathbb{R}^{3})} \lesssim  \| \nabla u \|_{L_{t}^{\infty} L_{x}^{2}}^{1/2} \| u \|_{L_{t,x}^{8}(I \times \mathbb{R}^{3})}^{9/2}.
\end{equation}
Then by the Sobolev embedding theorem,
\begin{equation}\label{2.11}
\| \int_{t_{0}}^{t} S(t - \tau)(0, u^{5}) d\tau \|_{L_{t,x}^{8}(I \times \mathbb{R}^{3})} \lesssim  \| \nabla u \|_{L_{t}^{\infty} L_{x}^{2}}^{1/2} \| u \|_{L_{t,x}^{8}(I \times \mathbb{R}^{3})}^{9/2}.
\end{equation}
Choosing $\epsilon(E)$ sufficiently small, $\epsilon(E) \sim E^{-1/14}$ will do, the proof is complete.
\end{proof}


\section{A Reduction of the solution}

Strichartz estimates imply that
\begin{equation}\label{3.2}
\| S(t)(u_{0}, u_{1}) \|_{L_{t,x}^{8}(\mathbb{R} \times \mathbb{R}^{3})} \lesssim E^{1/2}.
\end{equation}
Next, partition $\mathbb{R}$ into $\sim \frac{E^{4}}{\epsilon^{8}}$ subintervals $I_{j}$ such that
\begin{equation}\label{3.3}
\| S(t)(u_{0}, u_{1}) \|_{L_{t,x}^{8}(I_{j} \times \mathbb{R}^{3})} \leq \frac{\epsilon}{4},
\end{equation}
where $\epsilon \sim E^{-1/14}$. Then by the triangle inequality,
\begin{equation}\label{3.4}
\| u \|_{L_{t,x}^{8}(\mathbb{R} \times \mathbb{R}^{3})} \lesssim E^{4/7} \sup_{j} \| u \|_{L_{t,x}^{8}(I_{j} \times \mathbb{R}^{3})}.
\end{equation}
Therefore, to obtain a bound on the scattering size of $u$, it is enough to obtain a bound on $\| u \|_{L_{t,x}^{8}(I_{j} \times \mathbb{R}^{3})}$ that is uniform in $j$.\medskip

Fix $I_{j} = [a_{j}, b_{j}]$. If there exists some $t_{j} \in [a_{j}, b_{j}]$ such that
\begin{equation}\label{3.5}
\| S(t - t_{j})(u(t_{j}), u_{t}(t_{j})) \|_{L_{t,x}^{8}([t_{j}, b_{j}] \times \mathbb{R}^{3})} \leq \epsilon,
\end{equation}
then Theorems $\ref{t2.1}$ and $\ref{t2.2}$ imply
\begin{equation}\label{3.6}
\| u \|_{L_{t,x}^{8}(I_{j} \times \mathbb{R}^{3})} \lesssim A \langle t_{j} - a_{j} \rangle^{1/8} + \epsilon.
\end{equation}

For any $t \in \mathbb{R}$, Duhamel's principle implies that the solution to $(\ref{1.1})$ has the form
\begin{equation}\label{3.1}
S(t)(u_{0}, u_{1}) - \int_{0}^{t} S(t - \tau)(0, u^{5}) d\tau.
\end{equation}
Therefore, for any $t > t_{j}$,
\begin{equation}\label{3.7}
S(t - t_{j})(u(t_{j}), u_{t}(t_{j})) = S(t)(u_{0}, u_{1}) - \int_{0}^{t_{j}} S(t - \tau)(0, u^{5}) d\tau.
\end{equation}
By definition of $I_{j}$, $(\ref{3.3})$ implies
\begin{equation}\label{3.8}
\| S(t)(u_{0}, u_{1}) \|_{L_{t,x}^{8}([t_{j}, b_{j}] \times \mathbb{R}^{3})} \leq \frac{\epsilon}{4}.
\end{equation}
Therefore, to obtain $(\ref{3.5})$ it only remains to prove
\begin{equation}\label{3.9}
\| \int_{0}^{t_{j}} S(t - \tau)(0, u^{5}) d\tau \|_{L_{t,x}^{8}([t_{j}, b_{j}] \times \mathbb{R}^{3})} \leq \frac{3 \epsilon}{4}.
\end{equation}
Suppose without loss of generality that $a_{j} \geq 0$ and decompose
\begin{equation}\label{3.10}
\int_{0}^{t_{j}} S(t - \tau)(0, u^{5}) d\tau = \int_{0}^{t_{j} - T} S(t - \tau)(0, u^{5}) d\tau + \int_{t_{j} - T}^{t_{j}} S(t - \tau)(0, u^{5}) d\tau.
\end{equation}

\begin{lemma}\label{l3.1}
There exists $T \sim \frac{A^{2} E^{13/6}}{\epsilon^{4}}$ such that
\begin{equation}\label{3.11}
\| \int_{0}^{t_{j} - T} S(t - \tau)(0, u^{5}) d\tau \|_{L_{t,x}^{8}([t_{j}, b_{j}] \times \mathbb{R}^{3})} \leq \frac{\epsilon}{4},
\end{equation}
for any $t_{j} \in \mathbb{R}$.
\end{lemma}
\begin{proof}
For any $(t, x) \in [t_{j}, b_{j}] \times \mathbb{R}^{3}$,
\begin{equation}\label{3.12}
(\int_{0}^{t_{j} - T} S(t - \tau)(0, u^{5}) d\tau)(t,x) = \int_{0}^{t_{j} - T} \frac{1}{4 \pi(t - \tau)} \int_{|x - x'| = |t - \tau|} u^{5}(\tau, x') dS(x'),
\end{equation}
where $dS$ denotes the surface measure of a sphere in $\mathbb{R}^{3}$. See for example \cite{sogge1995lectures}. Also, by computing the energy flux, for any $R > 0$, $x_{0} \in \mathbb{R}^{3}$,
\begin{equation}\label{3.13}
\frac{d}{dt} \int_{|x - x_{0}| \leq R - t} [\frac{1}{2} |\nabla u(t,x)|^{2} + \frac{1}{2} |u_{t}(t,x)|^{2} + \frac{1}{6} |u(t,x)|^{6}] dx \leq -\frac{1}{2} \int_{|x - x_{0}| = R - t} |u(t,x)|^{6} d\sigma(x).
\end{equation}
Therefore, by conservation of energy, $(\ref{3.13})$, and H{\"o}lder's inequality,
\begin{equation}\label{3.14}
|\int_{0}^{t_{j} - T} \frac{1}{4 \pi(t - \tau)} \int_{|x - x'| = |t - \tau|} u^{5}(\tau, x') dS(x')| \lesssim E^{5/6} (\int_{0}^{t_{j} - T} \frac{1}{|t - \tau|^{4}})^{1/6} \lesssim \frac{E^{5/6}}{T^{1/2}}.
\end{equation}
Strichartz estimates, the energy identity, and Duhamel's principle also imply that
\begin{equation}\label{3.16}
\aligned
\| \int_{0}^{t_{j} - T} S(t - \tau)(0, u^{5}) d\tau \|_{L_{t,x}^{4}([t_{j}, b_{j}] \times \mathbb{R}^{3})} \lesssim \| (u(t_{j} - T), u_{t}(t_{j} - T)) \|_{\dot{H}^{1/2} \times \dot{H}^{-1/2}} \\ + \| (u_{0}, u_{1}) \|_{\dot{H}^{1/2} \times \dot{H}^{-1/2}} \lesssim E^{1/4} A^{1/2}.
\endaligned
\end{equation}
Interpolating $(\ref{3.14})$ and $(\ref{3.16})$,
\begin{equation}\label{3.17}
\| \int_{0}^{t_{j} - T} S(t - \tau)(0, u^{5}) d\tau \|_{L_{t,x}^{8}([t_{j}, b_{j}] \times \mathbb{R}^{3})} \lesssim A^{1/2} \frac{E^{13/24}}{T^{1/4}}.
\end{equation}
Taking $T \sim \frac{A^{2} E^{13/6}}{\epsilon^{4}}$ gives the estimate
\begin{equation}\label{3.18}
\| \int_{0}^{t_{j} - T} S(t - \tau)(0, u^{5}) d\tau \|_{L_{t,x}^{8}([t_{j}, b_{j}] \times \mathbb{R}^{3})} \leq \frac{\epsilon}{4}.
\end{equation}
\end{proof}

Therefore, to obtain uniform spacetime integral bounds on a solution to $(\ref{1.1})$, it only remains to prove that there exists $\mathcal T(T)$ such that for any $j$, if $b_{j} - a_{j} \geq \mathcal T(T)$, there exists $t_{j} \in I_{j}$, $0 < t_{j} - a_{j} \leq \mathcal T(T)$ that satisfies
\begin{equation}\label{3.19}
\| \int_{t_{j} - T}^{t_{j}} S(t - \tau)(0, u^{5}) d\tau \|_{L_{t,x}^{8}([t_{j}, b_{j}] \times \mathbb{R}^{3})} \leq \frac{\epsilon}{2}. 
\end{equation}
This will be the topic of the next section. The proof will utilize an interaction Morawetz estimate.

\section{Interaction Morawetz estimate}
\begin{theorem}\label{t4.1}
For $T \sim \frac{A^{2} E^{13/6}}{\epsilon^{4}}$ there exists $\mathcal T(T)$ such that for any $j$, if $b_{j} - a_{j} \geq \mathcal T(T)$, there exists $t_{j} \in I_{j}$, $0 < t_{j} - a_{j} \leq \mathcal T(T)$ that satisfies
\begin{equation}\label{5.1}
\| \int_{t_{j} - T}^{t_{j}} S(t - \tau)(0, u^{5}) d\tau \|_{L_{t,x}^{8}([t_{j}, b_{j}] \times \mathbb{R}^{3})} \leq \frac{\epsilon}{2}. 
\end{equation}
\end{theorem}
\begin{proof}
The proof uses an interaction Morawetz estimate. Choose $\phi \in C_{0}^{\infty}(\mathbb{R}^{3})$ that is supported on $|x| \leq 2$, $\phi(x) = 1$ on $|x| \leq 1$, $\phi(x) \geq 0$. Let $M_{R}(t)$ denote the Morawetz potential,
\begin{equation}\label{5.2}
M_{R}(t) = \int e(t,y) \phi(\frac{x - y}{R}) (x - y) \cdot \langle u_{t}, \nabla u \rangle dx dy + \int e(t,y) \phi(\frac{x - y}{R}) \langle u_{t}, u \rangle dx dy,
\end{equation}
where $R > 0$ is a fixed constant and $e(t, y)$ is the energy density
\begin{equation}\label{5.3}
e(t, y) = \frac{1}{2} u_{t}(t, y)^{2} + \frac{1}{2} |\nabla u(t, y)|^{2} + \frac{1}{6} |u(t,y)|^{6}.
\end{equation}
By direct calculation,
\begin{equation}\label{5.4}
\frac{d}{dt} M_{R}(t) = \int \nabla \cdot \langle u_{t}, \nabla u \rangle \phi(\frac{x - y}{R}) (x - y) \cdot \langle u_{t}, \nabla u \rangle
\end{equation}
\begin{equation}\label{5.5}
+ \int \nabla \cdot \langle u_{t}, \nabla u \rangle \phi(\frac{x - y}{R}) \langle u_{t}, u \rangle dx dy
\end{equation}
\begin{equation}\label{5.6}
+ \int e(t,y) \phi(\frac{x - y}{R}) (x - y) \cdot [\langle \nabla u, \Delta u \rangle - \langle u^{5}, \nabla u \rangle] dx dy
\end{equation}
\begin{equation}\label{5.7}
+ \int e(t,y) \phi(\frac{x - y}{R}) (x - y) \cdot \langle u_{t}, \nabla u_{t} \rangle dx dy
\end{equation}
\begin{equation}\label{5.8}
+ \int e(t,y) \phi(\frac{x - y}{R}) [\langle u, \Delta u \rangle - \langle u, u^{5} \rangle] dx dy
\end{equation}
\begin{equation}\label{5.9}
+ \int e(t,y) \phi(\frac{x - y}{R}) \langle u_{t}, u_{t} \rangle dx dy.
\end{equation}
Integrating by parts,
\begin{equation}\label{5.10}
(\ref{5.7}) = -\frac{3}{2} \int e(t,y) \phi(\frac{x - y}{R}) u_{t}^{2} dx dy - \frac{1}{2} \int e(t,y) \phi'(\frac{x - y}{R}) \frac{|x - y|}{R} u_{t}^{2} dx dy,
\end{equation}
so
\begin{equation}\label{5.11}
(\ref{5.7}) + (\ref{5.9}) = -\frac{1}{2} \int e(t,y) \phi(\frac{x - y}{R}) u_{t}^{2} dx dy - \frac{1}{2} \int e(t,y) \phi'(\frac{x - y}{R}) \frac{|x - y|}{R} u_{t}^{2} dx dy.
\end{equation}
Also integrating by parts,
\begin{equation}\label{5.12}
\aligned
(\ref{5.8}) = -\int e(t,y) \phi(\frac{x - y}{R}) |\nabla u|^{2} dx dy - \int e(t,y) \phi(\frac{x - y}{R}) u^{6} dx dy \\ + \frac{1}{2} \int e(t,y) \phi''(\frac{x - y}{R}) \frac{1}{R^{2}} u^{2} dx dy.
\endaligned
\end{equation}
Rewriting,
\begin{equation}\label{5.13}
\langle \Delta u, \nabla u \rangle = \partial_{k} \langle \partial_{j} u, \partial_{k} u \rangle - \frac{1}{2} \partial_{j} \langle \partial_{k} u, \partial_{k} u \rangle,
\end{equation}
and integrating by parts,
\begin{equation}\label{5.14}
\aligned
(\ref{5.6}) = \frac{1}{2} \int e(t,y) \phi(\frac{x - y}{R}) |\nabla u|^{2} dx dy + \frac{1}{6} \int e(t,y) \phi(\frac{x - y}{R}) u^{6} dx dy \\
+ \frac{1}{2} \int e(t,y) \phi'(\frac{x - y}{R}) \frac{(x - y)_{j} (x - y)_{k}}{|x - y| R} \langle \partial_{j} u, \partial_{k} u \rangle dx dy \\ - \frac{1}{2} \int e(t,y) \phi'(\frac{x - y}{R}) \frac{|x - y|}{R} |\nabla u|^{2} dx dy \\ + \frac{1}{6} \int e(t,y) \phi'(\frac{x - y}{R}) \frac{|x - y|}{R} u^{6} dx dy.
\endaligned
\end{equation}
Also, integrating by parts,
\begin{equation}\label{5.15}
\aligned
(\ref{5.4}) = \int \langle \partial_{k} u, u_{t} \rangle \phi(\frac{x - y}{R}) \delta_{jk} \langle \partial_{j} u, u_{t} \rangle dx dy \\ + \int \langle \partial_{k} u, u_{t} \rangle \phi'(\frac{x - y}{R}) \frac{(x - y)_{j} (x - y)_{k}}{|x - y| R} \langle \partial_{j} u, u_{t} \rangle dx dy.
\endaligned
\end{equation}
Finally, integrating by parts,
\begin{equation}\label{5.16}
\aligned
(\ref{5.5}) = \int \langle \partial_{k} u, u_{t} \rangle \phi'(\frac{x - y}{R}) \frac{(x - y)_{k}}{|x - y| R} \langle u, u_{t} \rangle dx dy.
\endaligned
\end{equation}
Therefore,
\begin{equation}\label{5.17}
\aligned
(\ref{5.4}) + (\ref{5.5}) + (\ref{5.6}) + (\ref{5.7}) + (\ref{5.8}) + (\ref{5.9}) \\ = -\int [\frac{1}{2} u_{t}^{2} + \frac{1}{2} |\nabla u|^{2} + \frac{1}{6} u^{6}] \phi(\frac{x - y}{R}) [\frac{1}{2} u_{t}^{2} + \frac{1}{2} |\nabla u|^{2} + \frac{2}{3} u^{6}] dx dy \\
+ \int \langle \partial_{k} u, u_{t} \rangle \phi(\frac{x - y}{R}) \delta_{jk} \langle \partial_{j} u, u_{t} \rangle dx dy \\ 
+ \int \langle \partial_{k} u, u_{t} \rangle \phi'(\frac{x - y}{R}) \frac{(x - y)_{j} (x - y)_{k}}{|x - y| R} \langle \partial_{j} u, u_{t} \rangle dx dy \\
- \frac{1}{2} \int e(t,y) \phi'(\frac{x - y}{R}) \frac{|x - y|}{R} u_{t}^{2} dx dy \\
+ \frac{1}{2} \int e(t,y) \phi''(\frac{x - y}{R}) \frac{1}{R^{2}} u^{2} dx dy \\
+ \frac{1}{2} \int e(t,y) \phi'(\frac{x - y}{R}) \frac{(x - y)_{j} (x - y)_{k}}{|x - y| R} \langle \partial_{j} u, \partial_{k} u \rangle dx dy \\ - \frac{1}{2} \int e(t,y) \phi'(\frac{x - y}{R}) \frac{|x - y|}{R} |\nabla u|^{2} dx dy \\ + \frac{1}{6} \int e(t,y) \phi'(\frac{x - y}{R}) \frac{|x - y|}{R} u^{6} dx dy \\
+  \int \langle \partial_{k} u, u_{t} \rangle \phi'(\frac{x - y}{R}) \frac{(x - y)_{k}}{|x - y| R} \langle u, u_{t} \rangle dx dy.
\endaligned
\end{equation}
Therefore,
\begin{equation}\label{5.18}
\aligned
\frac{1}{J} \int_{R_{0}}^{e^{J} R_{0}} \frac{1}{R} \frac{d}{dt} M_{R}(t) dR \\ = -\frac{1}{J} \int_{R_{0}}^{e^{J} R_{0}} \frac{1}{R} \int [\frac{1}{2} u_{t}^{2} + \frac{1}{2} |\nabla u|^{2} + \frac{1}{6} u^{6}] \phi(\frac{x - y}{R}) [\frac{1}{2} u_{t}^{2} + \frac{1}{2} |\nabla u|^{2} + \frac{2}{3} u^{6}] dx dy dR \\
+ \frac{1}{J} \int_{R_{0}}^{e^{J} R_{0}} \frac{1}{R} \int \langle \partial_{j} u, u_{t} \rangle \phi(\frac{x - y}{R}) \langle \partial_{j} u, u_{t} \rangle dx dy dR \\
+ O(\frac{1}{J} \int_{|x - y| \leq 2 e^{J} R_{0}} e(t,y) e(t,x) dx dy) \\ + O(\frac{1}{J} \int_{|x - y| \leq 2 e^{J} R_{0}} e(t,y) \frac{1}{|x - y|^{2}} u(t,x)^{2} dx dy).
\endaligned
\end{equation}
Therefore, by the fundamental theorem of calculus,
\begin{equation}\label{5.19}
\aligned
\int_{a_{j}}^{a_{j} + \mathcal T} \frac{1}{J} \int_{R_{0}}^{e^{J} R_{0}} \frac{1}{R} \int [\frac{1}{2} u_{t}^{2} + \frac{1}{2} |\nabla u|^{2} + \frac{1}{6} u^{6}] \phi(\frac{x - y}{R}) [\frac{1}{2} u_{t}^{2} + \frac{1}{2} |\nabla u|^{2} + \frac{2}{3} u^{6}] dx dy dR dt \\
- \int_{0}^{T} \frac{1}{J} \int_{R_{0}}^{e^{J} R_{0}} \frac{1}{R} \int |\nabla u| |u_{t}| \phi(\frac{x - y}{R}) |\nabla u| |u_{t}| dx dy dR dt \lesssim \frac{\mathcal T}{J} E^{2} + \frac{e^{J} R_{0}}{J} E^{2}.
\endaligned
\end{equation}
Now by positive definiteness argument, if we choose $\mathcal T = e^{J} R_{0}$,
\begin{equation}\label{5.20}
\aligned
\int_{a_{j}}^{a_{j} + \mathcal T} \frac{1}{J} \int_{R_{0}}^{e^{J} R_{0}} \frac{1}{R} \int [\frac{1}{2} (|u_{t}| - |\nabla u|)^{2} + \frac{1}{6} u^{6}] \phi(\frac{x - y}{R}) [\frac{1}{2} u_{t}^{2} + \frac{1}{2} |\nabla u|^{2} + \frac{2}{3} u^{6}] dx dy dR dt  \lesssim \frac{\mathcal T}{J} E^{2}.
\endaligned
\end{equation}

Inequality $(\ref{5.20})$ implies that there exists some $t_{j} \in [a_{j} + T, a_{j} + \mathcal T + T]$ and some $R_{0} \leq R \leq e^{J} R_{0}$ such that
\begin{equation}\label{5.21}
\int_{t_{j} - T}^{t_{j}} \int \int [\frac{1}{2}(|\nabla u| - |u_{t}|)^{2} + \frac{1}{6} u^{6}] \phi(\frac{x - y}{R}) [\frac{1}{2} u_{t}^{2} + \frac{1}{2} |\nabla u|^{2} + \frac{2}{3} u^{6}] dx dy dt \lesssim \frac{T}{J} E^{2}.
\end{equation}
Rewriting $(\ref{5.21})$ as a sum,
\begin{equation}\label{5.22}
\int_{t_{j} - T}^{t_{j}} \sum_{k \in \mathbb{Z}^{3}} (\int \chi(\frac{Rk - x}{R}) (\frac{1}{2} (|\nabla u| - |u_{t}|)^{2} + \frac{1}{6} u^{6} dx)) (\int \chi(\frac{Rk - x}{R}) (\frac{1}{2} u_{t}^{2} + \frac{1}{2} |\nabla u|^{2} + \frac{2}{3} u^{6} dx) dt \lesssim \frac{T}{J} E^{2},
\end{equation}
where $\chi$ is a smooth, compactly supported function, $\chi(x) \geq 0$ on the ball of radius $1$.

Also, by the Sobolev embedding theorem,
\begin{equation}\label{5.23}
\aligned
\| \chi(\frac{Rk - x}{R}) u \|_{L^{\infty}(\mathbb{R}^{3})} \lesssim \| \chi(\frac{Rk - x}{R}) u \|_{L^{6}}^{1/2} \| \nabla (\chi(\frac{Rk - x}{R}) u) \|_{L^{6}}^{1/2} \\
\lesssim \frac{1}{R^{1/2}} \| \chi(\frac{Rk - x}{R}) u \|_{L^{6}}^{1/2} \| \chi'(\frac{Rk - x}{R}) u \|_{L^{6}}^{1/2} +  \| \chi(\frac{Rk - x}{R}) u \|_{L^{6}}^{1/2} \| \chi(\frac{Rk - x}{R}) \nabla u \|_{L^{6}}^{1/2}.
\endaligned
\end{equation}
Therefore, by the support properties of $\chi$ and $\chi'$, $(\ref{5.23})$ implies
\begin{equation}\label{5.24}
(\sum_{k \in \mathbb{Z}^{3}} \| \chi(\frac{Rk - x}{R}) u \|_{L^{\infty}(\mathbb{R}^{3})}^{6})^{1/6} \lesssim \frac{1}{R^{1/2}} \| u \|_{L^{6}(\mathbb{R}^{3})} + \| \nabla u \|_{L^{6}(\mathbb{R}^{3})}^{1/2} \| u \|_{L^{6}(\mathbb{R}^{3})}^{1/2} \lesssim \frac{1}{R^{1/2}} E^{1/6} + E^{1/12} A^{1/2}.
\end{equation}
Taking $R \gtrsim \frac{E^{1/3}}{A}$,
\begin{equation}\label{5.25}
(\sum_{k \in \mathbb{Z}^{3}} \| \chi(\frac{Rk - x}{R}) u \|_{L^{\infty}(\mathbb{R}^{3})}^{6})^{1/6} \lesssim E^{1/12} A^{1/2}.
\end{equation}
Meanwhile, by $(\ref{5.21})$,
\begin{equation}\label{5.26}
(\int_{t_{j} - T}^{t_{j}} \sum_{k \in \mathbb{Z}^{3}} \| \chi(\frac{Rk - x}{R}) u \|_{L^{6}(\mathbb{R}^{3})}^{12} dt)^{1/12} \lesssim \frac{T^{1/12}}{J^{1/12}} E^{1/6}.
\end{equation}
Interpolating $(\ref{5.25})$ and $(\ref{5.26})$,
\begin{equation}\label{5.27}
\| u \|_{L_{t}^{24} L_{x}^{9}([t_{j} - T, t_{j}] \times \mathbb{R}^{3})} \lesssim \frac{T^{1/24}}{J^{1/24}} E^{1/8} A^{1/4}.
\end{equation}
Also by $(\ref{1.12})$,
\begin{equation}\label{5.28}
\| u \|_{L_{t}^{\infty} L_{x}^{2}([t_{j} - T, t_{j}] \times \mathbb{R}^{3})} \lesssim A.
\end{equation}
Interpolating $(\ref{5.27})$ and $(\ref{5.28})$
\begin{equation}\label{5.29}
\| u \|_{L_{t}^{27} L_{x}^{\frac{162}{25}}([t_{j} - T, t_{j}] \times \mathbb{R}^{3})} \lesssim \frac{T^{1/27}}{J^{1/27}} E^{1/9} A^{1/3}.
\end{equation}

Choosing $\delta > 0$ sufficiently small such that $\delta^{3} E \lesssim \epsilon$, $\| u \|_{L_{t}^{27} L_{x}^{\frac{162}{25}}([t_{j} - T, t_{j}] \times \mathbb{R}^{3})} \leq \delta$ implies
\begin{equation}\label{5.30}
\| \int_{t_{j} - T}^{t_{j}} S(t - \tau)(0, u^{5}) d\tau \|_{L_{t,x}^{8}(\mathbb{R} \times \mathbb{R}^{3})} \lesssim \epsilon.
\end{equation}
Indeed, by Strichartz estimates and conservation of energy, if $(p, q)$ is the $\dot{H}^{1}$-admissible pair $(p, q) = (\frac{9}{4}, 54)$,
\begin{equation}\label{5.31}
\| u \|_{L_{t}^{p} L_{x}^{q}} \lesssim E^{1/2} + \delta^{3} \| u \|_{L_{t}^{p} L_{x}^{q}}^{2},
\end{equation}
which implies $\| u \|_{L_{t}^{p} L_{x}^{q}([t_{j} - T, t_{j}] \times \mathbb{R}^{3})} \lesssim E^{1/2}$. Plugging this fact into $(\ref{5.30})$ gives the appropriate estimate.

Doing some algebra with $(\ref{5.29})$, since $\epsilon \sim E^{-1/14}$ and $T \sim \frac{A^{2} E^{13/6}}{\epsilon^{4}}$,
\begin{equation}\label{5.31.1}
\frac{T^{1/27}}{J^{1/27}} E^{1/9} A^{1/3} \sim \frac{A^{11/27} E^{31/162}}{J^{1/27} \epsilon^{4/27}} \lesssim \frac{\epsilon^{1/3}}{E^{1/3}},
\end{equation}
and therefore,
\begin{equation}\label{5.31.2}
\frac{A^{11/27} E^{85/162}}{\epsilon^{13/27}} \sim A^{11/27} E^{85/162} E^{13/378} \lesssim J^{1/27},
\end{equation}
implies that
\begin{equation}\label{5.32}
\| \int_{t_{j} - T}^{t_{j}} S(t - \tau)(0, u^{5}) \|_{L_{t,x}^{8}([t_{j}, b_{j}] \times \mathbb{R}^{3})} \leq \frac{\epsilon}{2}.
\end{equation}
This proves the theorem.
\end{proof}

Setting $\mathcal T = \exp(C E^{85/6} E^{13/14} A^{11})$ for some constant $C$, $(\ref{3.4})$ and $(\ref{3.6})$ imply
\begin{equation}\label{5.33}
\| u \|_{L_{t,x}^{8}(\mathbb{R} \times \mathbb{R}^{3})} \lesssim E^{4/7} A \exp(C E^{85/6} E^{13/14} A^{11}),
\end{equation}
which proves Theorem $\ref{t1.2}$.

\section*{Acknowledgement}
I am grateful to Anudeep Kumar Arora and Jason Murphy for many helpful discussions related to this work.

\bibliography{biblio}
\bibliographystyle{alpha}

\end{document}